\documentclass[a4paper,10pt]{amsart}
\usepackage[utf8]{inputenc}
\usepackage{amsxtra}
\usepackage{amsopn}
\usepackage{amsmath,amsthm,amssymb}
\usepackage{amscd}
\usepackage{mathtools}
\usepackage{amsfonts}
\usepackage{latexsym}
\usepackage{verbatim}
\usepackage{MnSymbol}
\usepackage{xcolor}
\usepackage{cases}
\usepackage{tikz-cd}
\usepackage{cases}

\newcommand{\lr}{\longrightarrow}

\let\c\overline
\theoremstyle{plain}
\newtheorem{theorem}{Theorem}[section]
\newtheorem*{theorem*}{Theorem \ref{thm-main}}
\newtheorem*{theorem**}{Theorem \ref{thm-dol-11-dim-char}}
\newtheorem*{proposition*}{Proposition \ref{prop-alm-kahl}}

\newtheorem{lemma}[theorem]{Lemma}
\newtheorem{proposition}[theorem]{Proposition}
\newtheorem{corollary}[theorem]{Corollary}
\newtheorem{remark}[theorem]{Remark}

\newtheorem{question}[theorem]{Question}
\newtheorem*{question*}{Question}
\newtheorem*{mt*}{Main Theorem}
\newcommand\C{{\mathbb C}}

\newcommand\R{{\mathbb R}}

\newcommand{\del}{{\partial}}
\newcommand{\delbar}{{\overline{\del}}}
\newcommand{\mubar}{{\overline{\mu}}}
\newcommand{\cinf}{\mathcal{C}^\infty}
\renewcommand{\H}{\mathcal{H}}

\DeclareMathOperator{\vol}{Vol}

\DeclareMathOperator{\real}{Re}
\DeclareMathOperator{\img}{Im}

\DeclareMathOperator{\End}{End}
\DeclareMathOperator{\id}{id}

\let\c\overline
\let\phi\varphi
\title[Almost Hermitian compact quotients of Lie groups]{Dolbeault Harmonic $(1,1)$-forms on $4$-dimensional compact quotients of Lie Groups with a left invariant almost Hermitian structure}
\author{Riccardo Piovani}
\address{Dipartimento di Scienze Matematiche, Fisiche e Informatiche\\
Unit\`{a} di Matematica e Informatica\\
Universit\`{a} degli Studi di Parma\\
Parco Area delle Scienze 53/A \\
43124 Parma, Italy}
\email{riccardo.piovani@unipr.it}

\thanks{\newline 
The author is partially supported by GNSAGA of INdAM}

\begin{document}
\maketitle
%\tableofcontents

\begin{abstract} 
We study Dolbeault harmonic $(1,1)$-forms on compact quotients $M=\Gamma\backslash G$ of $4$-dimensional Lie groups $G$ admitting a left invariant almost Hermitian structure $(J,\omega)$. In this case, we prove that the space of Dolbeault harmonic $(1,1)$-forms on $(M,J,\omega)$ has dimension $b^-+1$ if and only if there exists a left invariant anti self dual $(1,1)$-form $\gamma$ on $(G,J)$ satisfying $id^c\gamma=d\omega$. Otherwise, its dimension is $b^-$. In this way, we answer to a question by Zhang.
\end{abstract}

\section{Introduction}
Let $(M,J)$ be an almost complex manifold of real dimension $2n$. Given, on $(M,J)$, an almost Hermitian metric $g$, with fundamental form $\omega$, the triple $(M,J,\omega)$ will be called an almost Hermitian manifold. The exterior derivative decomposes into
\[
d=\mu+\del+\delbar+\c\mu.
\]
By the map $*:A^{p,q}\to A^{n-q,n-p}$, we denote the $\C$-linear extension of the real Hodge $*$ operator. 
We set
$\del^*:=-*\delbar*$ and $\delbar^*:=-*\del*$, which are the $L^2$ formal adjoint operators respectively of $\del$ and $\delbar$, i.e., they are $L^2$ adjoint when $M$ is compact. Recall that
\begin{gather*}
\Delta_{\delbar}:=\delbar\delbar^*+\delbar^*\delbar
\end{gather*}
is the Dolbeault Laplacian, which is a formally self adjoint elliptic operator of order $2$. 
We set 
\[
\H^{p,q}_{\delbar}:=\ker\Delta_{\delbar}\cap A^{p,q}
\]
to be the space of Dolbeault harmonic $(p,q)$-forms.  
If $M$ is compact, it is well known that the dimension
\[
h^{p,q}_{\delbar}:=\dim_\C\H^{p,q}_{\delbar},
\]
is finite. If the almost complex structure $J$ is integrable, i.e., if $(M,J)$ is a complex manifold, then the numbers $h^{p,q}_{\delbar}$ depend only on the complex structure and not on the metric, being the dimensions of the $(p,q)$-Dolbeault cohomology spaces by Hodge theory.

Whether or not the numbers $h^{p,q}_{\delbar}$ depend on the choice of the almost Hermitian metric $\omega$ on a given compact almost Hermitian manifold $(M,J,\omega)$ is a question by Kodaira and Spencer, which appeared as Problem 20 in Hirzebruch’s 1954 problem list \cite{Hi}.
Recently, in \cite{HZ} (cf. \cite{HZ2}), Holt and Zhang proved that the number $h^{0,1}_{\delbar}$ depends on the choice of the metric on a given compact almost complex $4$-manifold, answering the question of Kodaira and Spencer. They also proved that on a compact almost K\"ahler $4$-manifold it holds that $h^{1,1}_\delbar=b^-+1$, see \cite[Proposition 6.1]{HZ}.
Later, in \cite{TT}, Tardini and Tomassini proved that the number $h^{1,1}_{\delbar}$ also depends on the choice of the metric. In particular, on a compact almost Hermitian $4$-manifold $(M,J,\omega)$,
they proved that $h^{1,1}_{\delbar}=b^-$ if $\omega$ is strictly locally conformally almost K\"ahler, and $h^{1,1}_{\delbar}=b^-+1$ if $\omega$ is globally conformally almost K\"ahler, see \cite[Theorem 3.7]{TT}. 
Moreover, very recently, in \cite[Theorem 3.1]{Ho} Holt proved that $h^{1,1}_{\delbar}=b^-+1$ and $h^{1,1}_{\delbar}=b^-$ are the only two possible options on a compact almost Hermitian 4-manifold. As a corollary, he obtained the following
\begin{theorem}[{\cite[Corollary 3.2]{Ho}}]\label{thm-dol-11-dim-char}
Let $(M,J,\omega)$ be a compact almost Hermitian $4$-manifold. Assume that $\omega$ is Gauduchon, i.e., $\del\delbar\omega=0$. Then, $h^{1,1}_\delbar(M,J,\omega)=b^-+1$ if and only if there exists an anti self dual $(1,1)$-form $\gamma\in A^{1,1}(M,J)$ satisfying
\[
id^c\gamma=d\omega.
\]
Otherwise, $h^{1,1}_\delbar(M,J,\omega)=b^-$.
\end{theorem}
Recall that if $\omega$ is Gauduchon, then the space of Dolbeault harmonic $(1,1)$-forms is expressed by
\[
\H^{1,1}_\delbar=\{f\omega+\gamma\,|\,f\in\C,\ *\gamma=-\gamma,\ fd\omega=id^c\gamma\}.
\]
See \cite{piovani-parma} for a self contained survey of the above results.
See \cite{Ho,PT4} for similar results about Bott-Chern harmonic forms on compact almost Hermitian $4$-manifolds.
See also \cite{cattaneo-tardini-tomassini,piovani-tardini,tardini-tomassini-dim6} for other interesting results concerning Dolbeault and Bott-Chern harmonic forms on compact almost Hermitian manifolds of higher dimension.

Note that, in the integrable case, it is well known that a compact complex surface $(M,J)$ admits a K\"ahler metric if and only if $b^1$ is even, and $b^1$ is even if and only if $h^{1,1}_{\delbar}=b^-+1$, see e.g., \cite{BPV}. Looking at \cite[Proposition 6.1]{HZ} and \cite[Theorem 3.7]{TT}, on a $4$-manifold endowed with a non integrable almost Hermitian structure, one might expect that $h^{1,1}_{\delbar}=b^-+1$ if and only if the metric is globally conformally almost K\"ahler, see \cite[Question 3.3]{Ho}.
However, in the non integrable case, it might happen that $h^{1,1}_{\delbar}=b^-+1$ when the almost Hermitian metric is not globally conformally almost K\"ahler. Indeed, in \cite{PT5}, Tomassini and the author of the present paper proved that $h^{1,1}_{\delbar}=b^-+1$ on an explicit example of a $4$-dimensional compact almost complex manifold endowed with a non globally conformally almost K\"ahler metric.

Nonetheless, it is reasonable to expect $h^{1,1}_\delbar$ to detect almost K\"ahlerness in some other sense, as in the integrable case. Indeed, Zhang recently raised the following
\begin{question}[{\cite[Section 2.5]{zhang-parma}}]\label{question-zhang}
Let $(M,J)$ be a compact almost complex $4$-manifold. Is it true that $(M,J)$ admits an almost K\"ahler metric if and only if
\[
\underset{\omega \text{ almost Hermitian}}{\max} h^{1,1}_\delbar(M,J,\omega)=b^-+1
\]
holds?
\end{question}

In this paper, we show an improvement of Theorem \ref{thm-dol-11-dim-char} in the setting of compact quotients of Lie groups admitting a left invariant almost Hermitian structure. Our primary tool is the symmetrization process, following the ideas of \cite{belgun,fino-grantcharov}. The main result of this note is the next theorem.
\begin{theorem*}%\label{thm-main}
Let $G$ be a $4$-dimensional Lie group and let $\Gamma$ be a discrete subgroup such that $M=\Gamma\backslash G$ is compact. Assume that $(J,\omega)$ is a left invariant almost Hermitian structure on $G$. Then, $h^{1,1}_\delbar(M,J,\omega)=b^-+1$ if and only if there exists a left invariant anti self dual $(1,1)$-form $\gamma\in A^{1,1}(G,J)$ satisfying
\[
id^c\gamma=d\omega.
\]
Otherwise, $h^{1,1}_\delbar(M,J,\omega)=b^-$.
\end{theorem*}
We use an application of Theorem \ref{thm-main}, i.e., Corollary \ref{cor-char}, in many successive examples of left invariant almost Hermitian structures on compact quotients of solvable Lie groups, namely on the primary and secondary Kodaira surfaces, on the Inoue surfaces of type $S_M$, on the Hyperelliptic surfaces and on the $4$-dimensional nilmanifold which does not admit complex structures.

We remark that Theorem \ref{thm-main} is particularly useful when proving that $h^{1,1}_\delbar=b^-$ just having to show that there exist no left invariant anti self dual $(1,1)$-forms $\gamma\in A^{1,1}(G,J)$ satisfying $id^c\gamma=d\omega$, whereas, if one only uses Theorem \ref{thm-dol-11-dim-char}, one should prove that there exist no anti self dual $(1,1)$-forms $\gamma\in A^{1,1}(M,J)$ satisfying $id^c\gamma=d\omega$. In short, Theorem \ref{thm-main} allows us to just consider left invariant anti self dual $(1,1)$-forms when computing $h^{1,1}_\delbar$.

Furthermore, thanks to Corollary \ref{cor-char}, we are also able to answer Question \ref{question-zhang} negatively. In fact, both on the secondary Kodaira surface and on the Inoue surface of type $S_M$, we build a left invariant, non integrable almost complex structure which admits almost Hermitian metrics with $h^{1,1}_{\delbar}=b^-+1$, but does not admit almost K\"ahler metrics since $b^2=0$; see Propositions \ref{prop-sec-kod} and \ref{prop-inoue}. Here, Corollary \ref{cor-char} just plays the role of helping us to find such almost Hermitian metrics with $h^{1,1}_{\delbar}=b^-+1$.

The paper is organized in the following way. In section \ref{sec-preliminaries}, we introduce some standard preliminaries of almost complex and almost Hermitian geometry. In section \ref{sec-lie-groups}, we recall the symmetrization process and prove our main result, Theorem \ref{thm-main}, and  Corollary \ref{cor-char}. In section \ref{sec-applications} we apply Corollary \ref{cor-char} to many examples of left invariant almost Hermitian structures on compact quotients of solvable Lie groups. In this way, we also answer Question \ref{question-zhang}.

\medskip\medskip
\noindent{\em Acknowledgments.}
The author would like to thank sincerely Tom Holt, Nicoletta Tardini, Adriano Tomassini and Weiyi Zhang for useful and interesting conversations, as well as the referee for reading carefully the paper and for many valuable suggestions.

\section{Preliminaries}\label{sec-preliminaries}

Throughout this paper, we will only consider connected manifolds without boundary.
Let $(M,J)$ be an almost complex manifold of dimension $2n$, i.e., a $2n$-differentiable manifold endowed with an almost complex structure $J$, that is $J\in\End(TM)$ and $J^2=-\id$. The complexified tangent bundle $T_{\C}M=TM\otimes\C$ decomposes into the two eigenspaces of $J$ associated to the eigenvalues $i,-i$, which we denote respectively by $T^{1,0}M$ and $T^{0,1}M$, giving
\begin{equation*}
T_{\C}M=T^{1,0}M\oplus T^{0,1}M.
\end{equation*}
We denote by $\Lambda^{1,0}M$ and $\Lambda^{0,1}M$ the dual vector bundles of $T^{1,0}M$ and $T^{0,1}M$, respectively, and set
\begin{equation*}
\Lambda^{p,q}M=\bigwedge^p\Lambda^{1,0}M\wedge\bigwedge^q\Lambda^{0,1}M
\end{equation*}
to be the vector bundle of $(p,q)$-forms. Let $A^{p,q}=A^{p,q}(M,J)=\Gamma(M,\Lambda^{p,q}M)$ be the space of smooth sections of $\Lambda^{p,q}M$. We denote by $A^k=A^k(M)=\Gamma(M,\Lambda^{k}M)$ the space of $k$-forms. Note that $\Lambda^{k}M\otimes\C=\bigoplus_{p+q=k}\Lambda^{p,q}M$.

Let $f\in\cinf(M,\C)$ be a smooth function on $M$ with complex values. Its differential $df$ is contained in $A^1\otimes\C=A^{1,0}\oplus A^{0,1}$. On complex 1-forms, the exterior derivative acts as
\[
d:A^1\otimes\C\to A^2\otimes\C=A^{2,0}\oplus A^{1,1}\oplus A^{0,2}.
\]
 Therefore, it turns out that the exterior derivative operates on $(p,q)$-forms as
\begin{equation*}
d:A^{p,q}\to A^{p+2,q-1}\oplus A^{p+1,q}\oplus A^{p,q+1}\oplus A^{p-1,q+2},
\end{equation*}
where we denote the four components of $d$ by
\begin{equation*}
d=\mu+\del+\delbar+\c\mu.
\end{equation*}
From the relation $d^2=0$, we derive
\begin{equation*}
\begin{cases}
\mu^2=0,\\
\mu\del+\del\mu=0,\\
\del^2+\mu\delbar+\delbar\mu=0,\\
\del\delbar+\delbar\del+\mu\c\mu+\c\mu\mu=0,\\
\delbar^2+\c\mu\del+\del\c\mu=0,\\
\c\mu\delbar+\delbar\c\mu=0,\\
\c\mu^2=0.
\end{cases}
\end{equation*}
We also define the operator $d^c:=J^{-1}dJ$. It is a straightforward computation to show that
\[
d^c=i(\mu-\del+\delbar-\c\mu).
\]

Let $(M,J)$ be an almost complex manifold. If the almost complex structure $J$ is induced from a complex manifold structure on $M$, then $J$ is called integrable. Recall that $J$ being integrable is
equivalent to the exterior derivative decomposing into $d=\del+\delbar$.

A Riemannian metric on $M$ for which $J$ is an isometry is called almost Hermitian.
Let $g$ be an almost Hermitian metric, the $2$-form $\omega$ such that
\begin{equation*}
\omega(u,v)=g(Ju,v)\ \ \forall u,v\in\Gamma(TM)
\end{equation*}
is called the fundamental form of $g$. We will call $(M,J,\omega)$ an almost Hermitian manifold.
We denote by $h$ the Hermitian extension of $g$ on the complexified tangent bundle $T_\C M$, and by the same symbol $g$ the $\C$-bilinear symmetric extension of $g$ on $T_\C M$. Also denote by the same symbol $\omega$ the $\C$-bilinear extension of the fundamental form $\omega$ of $g$ on $T_\C M$. 
Thanks to the elementary properties of the two extensions $h$ and $g$, we may want to consider $h$ as a Hermitian operator
$T^{1,0}M\times T^{1,0}M\to\C$ and $g$ as a $\C$-bilinear operator $T^{1,0}M\times T^{0,1}M\to\C$.
Recall that it holds $h(u,v)=g(u,\bar{v})$ for all $u,v\in \Gamma(T^{1,0}M)$.

Let $(M,J,\omega)$ be an almost Hermitian manifold of real dimension $2n$. Extend $h$ on $(p,q)$-forms and denote the Hermitian inner product by $\langle\cdot,\cdot\rangle$.
Let $*:A^{p,q}\lr A^{n-q,n-p}$ the $\C$-linear extension of the standard Hodge $*$ operator on Riemannian manifolds with respect to the volume form $\vol=\frac{\omega^n}{n!}$.
Then the operators
\begin{equation*}
d^*=-*d*,\ \ \ \mu^*=-*\c\mu*,\ \ \ \del^*=-*\delbar*,\ \ \ \delbar^*=-*\del*,\ \ \ \c\mu^*=-*\mu*,
\end{equation*}
are the $L^2$ formal adjoint operators respectively of $d,\mu,\del,\delbar,\c\mu$. Recall that 
\[
\Delta_{d}=dd^*+d^*d
\]
 is the Hodge Laplacian, and, as in the integrable case, set 
\begin{equation*}
\Delta_{\delbar}=\delbar\delbar^*+\delbar^*\delbar,
\end{equation*}
as the $\delbar$, or Dolbeault, Laplacian.

If $M$ is compact, then we easily deduce the following relations
\begin{equation*}%\label{eq-char-harm-forms}
\begin{cases}
\Delta_{d}=0\ &\iff\ d=0,\ d*=0,\\
\Delta_{\delbar}=0\ &\iff\ \delbar=0,\ \del*=0,
\end{cases}
\end{equation*}
which characterize the spaces of harmonic forms
\begin{equation*}
\H^{k}_{d},\ \ \ \H^{p,q}_{\delbar},
\end{equation*}
defined as the spaces of forms which are in the kernel of the associated Laplacian.
These Laplacians are elliptic operators on the almost Hermitian manifold $(M,J,\omega)$, implying that the spaces of harmonic forms are finite dimensional when the manifold is compact. Denote by $\H^{p,q}_d$ the space $\big(\H^{p+q}_d\otimes\C\big)\cap A^{p,q}$, and by
\begin{equation*}
b^{k},\ \ \ h^{p,q}_d,\ \ \ h^{p,q}_{\delbar}
\end{equation*}
respectively the real dimension of $\H^k_d$, which is a topological invariant, and the complex dimensions of $\H^{p,q}_{d}$, $\H^{p,q}_{\delbar}$, which are almost Hermitian invariants.
%\begin{remark}
%By equation \eqref{bc-a-duality}, note that $*\H^{p,q}_{BC}=\H^{n-q,n-p}_{A}$ and $*\H^{p,q}_{A}=\H^{n-q,n-p}_{BC}$. In the following, we will study only the spaces $\H^{p,q}_{BC}$ on an almost complex manifolds; this is sufficient to describe also the spaces $\H^{p,q}_{A}$.
%\end{remark}

Let us focus for a moment on real dimension $4$.
Let $(M,g)$ be a compact oriented Riemannian manifold of real dimension 4, and set 
\[
\Lambda^-=\{\alpha\in \Lambda^2M\,:\,*\alpha=-\alpha\}
\]
to be the bundle of anti self dual $2$-forms. 
Denote by $A^-=\Gamma(M,\Lambda^-)$ the space of smooth anti self dual $2$-forms, and by
\[
\H^-=\{\alpha\in A^-\,:\,\Delta_d\alpha=0\},
\]
the subspace of harmonic anti self dual $2$-forms.
Set $b^-=\dim_{\R}\H^-$.
Note that $b^-$ is metric independent: see \cite[Chapter 1]{DK} for its topological meaning.

Let $(M,J,\omega)$ be an almost Hermitian manifold of real dimension $4$. Note that the space of anti self dual complex valued $2$-forms $A^-\otimes\C$ is indeed a subspace of $A^{1,1}$, which will be denoted by $A^-_{\C}$. Furthermore, the space $\H^-\otimes\C$ is indeed a subspace of $\H^{1,1}_d$, and will be denoted by $\H^-_{\C}$.

Recall that, on a given almost Hermitian manifold $(M,J,\omega)$ of dimension $2n$, the almost Hermitian metric is called Gauduchon if $\del\delbar\omega^{n-1}=0$, or equivalently if $dd^c\omega^{n-1}=0$, or equivalently if $d^*\theta=0$, where $\theta$ is the Lee form of $\omega$, uniquely determined by
\[
d\omega^{n-1}=\theta\wedge\omega^{n-1},
\]
thanks to the Lefschetz isomorphism.
Let us recall the following fundamental result by Gauduchon, \cite{Ga}: given a compact almost Hermitian $2n$-manifold $(M,J,\tilde\omega)$, there always exists a Gauduchon metric $\omega=e^t\tilde\omega$ in the conformal class of $\tilde\omega$, with $t\in\cinf(M)$, which is unique up to homothety for $n>1$.

\begin{remark} 
We point out that $h^{1,1}_\delbar$ is a conformal invariant of almost Hermitian metrics on a given almost complex $4$-manifold, see \cite[Lemma 3.1]{TT}.
Therefore, on a given almost Hermitian $4$-manifold $(M,J,\omega)$, in order to compute $h^{1,1}_\delbar(M,J,\omega)$, we can always choose to work with a Gauduchon metric $\tilde\omega$ which is conformal to $\omega$ and calculate $h^{1,1}_\delbar(M,J,\tilde\omega)$, which is equal to $h^{1,1}_\delbar(M,J,\omega)$.
\end{remark}

Finally, we will need a local formula for the Hodge $*$ operator. Let us recall it. Let $(M,J,\omega)$ be an almost Hermitian manifold of real dimension $2n$. Choose a local frame $\beta^1,\dots,\beta^n$ of $A^{1,0}$. We can write locally
\begin{equation*}
\omega=i\sum_{i,j=1}^ng_{i\c{j}}\beta^i\wedge\c{\beta}^j,
\end{equation*}
where $g_{i\c{j}}=\c{g_{j\c{i}}}$.
Let $\phi,\psi\in A^{p,q}$ be locally written as
\begin{equation*}
\phi=\sum\phi_{A_p\c{B}_q}\beta^{A_p\c{B}_q},\ \ \ \psi=\sum\psi_{A_p\c{B}_q}\beta^{A_p\c{B}_q},
\end{equation*}
with
\[
A_p=(a_1,\ldots,a_p),\qquad 
B_q=(b_1,\ldots,b_q)
\]
multi-indices of length $p$, $q$ respectively, such that
$1\le a_1<\cdots <a_p\le n$ and $1\le b_1<\cdots <b_q\le n$.
Given the matrix $(g^{\c{i}j})=(g_{i\c{j}})^{-1}$, we set
\[
\psi^{\c{A}_pB_q}=\sum g^{\c{ a}_1\gamma_1}\cdots g^{\c{ a}_p\gamma_p}
g^{\c{\lambda}_1 b_1}\cdots g^{\c{\lambda}_q b_q}\psi_{\gamma_1\ldots \gamma_p\c{\lambda}_1\ldots \c{\lambda}_q},
\]
such that the local formula for the pointwise Hermitian product $\langle\cdot,\cdot\rangle$ defined on $(p,q)$-forms is given by
\[
\langle\phi,\psi\rangle=\sum\phi_{A_p\c{B}_q}\c{\psi^{\c{A}_pB_q}}.
\]
We consider the volume form
\begin{equation*}
\frac{\omega^n}{n!}=i^n\det(g_{i\c{j}})\beta^{1\c1\dots n\c{n}}=i^n(-1)^\frac{n(n-1)}2\det(g_{i\c{j}})\beta^{1\dots n\c1\dots\c{n}}.
\end{equation*}
Recall that $*:A^{p,q}\to A^{n-q,n-p}$ is defined by the relation
\[
\phi\wedge*\c\psi=\langle\phi,\psi\rangle\frac{\omega^n}{n!}\ \ \ \forall\phi,\psi\in A^{p,q}.
\] 
Therefore, the local formula for $*\c\psi$ is
\begin{equation*}
*\c\psi=i^n(-1)^\frac{n(n-1)}2\det(g_{i\c{j}})\sum\epsilon_{A_p\c{B}_q}\c{\psi^{\c{A}_pB_q}}\beta^{A_n\setminus A_p\c{B}_n\setminus\c{B}_q},
\end{equation*}
where $\epsilon_{A_p\c{B}_q}$ is the sign of the permutation sending $(1,\dots,n,\c1,\dots,\c{n})$ to $(A_p,\c{B}_q,A_n\setminus A_p,\c{B}_n\setminus\c{B}_q)$. Since $*_\omega\c\psi=\c{(*_\omega\psi)}$, it follows
\begin{equation*}
*\psi=i^n(-1)^{n+\frac{n(n-1)}2+(n-p)(n-q)}\det(g_{i\c{j}})\sum\epsilon_{A_p\c{B}_q}{\psi^{\c{A}_pB_q}}\beta^{{B}_n\setminus{B}_q \c{A}_n\setminus \c{A}_p}.
\end{equation*}
Note that $(-1)^{n+\frac{n(n-1)}2+(n-p)(n-q)}=(-1)^{\frac{n(n-1)}2+pq-n(p+q)}$, therefore
\begin{equation}\label{eq-local-hodge-star}
*\psi=i^n(-1)^{\frac{n(n-1)}2+pq-n(p+q)}\det(g_{i\c{j}})\sum\epsilon_{A_p\c{B}_q}{\psi^{\c{A}_pB_q}}\beta^{{B}_n\setminus{B}_q \c{A}_n\setminus \c{A}_p}.
\end{equation}

\section{Invariant structures on compact quotients of Lie groups}\label{sec-lie-groups}
In this section we study compact quotients of Lie groups admitting a left invariant almost Hermitian structure, and focus on real dimension $4$ in order to find a characterization which describes the dimension of the space of Dolbeault harmonic $(1,1)$-forms.

We start by recalling the following fundamental result about compact quotients of Lie groups due to Milnor.

\begin{lemma}[{\cite[Lemma 6.2]{milnor}}]
If the Lie group $G$ admits a discrete subgroup $\Gamma$ with compact quotient, then $G$ admits a bi-invariant volume form $\nu$.
\end{lemma}

Now, let us introduce the symmetrization process following \cite[Theorem 2.1]{fino-grantcharov}, \cite[Theorem 7]{belgun} and \cite[p. 192]{ugarte}.

Let $G$ be a Lie group and let $\Gamma$ be a discrete subgroup such that $M=\Gamma\backslash G$ is compact. Let $\nu$ be a volume form on $M$ induced by a bi-invariant volume form on $G$. Denote by $\mathfrak{g}$ the Lie algebra of $G$. Given any covariant $k$-tensor field $T:\Gamma(M,TM)^k\to\cinf(M)$ on $M$, we define a covariant $k$-tensor field $T_\nu:\mathfrak{g}^k\to\R$ on $\mathfrak{g}$ by
\[
T_\nu(X_1,\dots,X_k)=\frac1{\nu(M)}\int_MT(X_1,\dots,X_k)\nu,\ \ \ \forall X_1,\dots,X_k\in\mathfrak{g},
\]
where $\nu(M)=\int_M\nu$ denotes the volume of $M$ with respect to the measure induced by $\nu$. Note that, if $T$ is left invariant, then $T_\nu=T$. The symmetrization operator has the following fundamental property. We recall the proof here for completeness.

\begin{theorem}[{\cite[Theorem 2.1]{fino-grantcharov}, cf. \cite[Theorem 7]{belgun}}]\label{thm-belgun}
Let $G$ be an $n$-dimensional Lie group and let $\Gamma$ be a discrete subgroup such that $M=\Gamma\backslash G$ is compact. Let $\nu$ be a volume form on $M$ induced by a bi-invariant volume form on $G$. If $\alpha\in A^k(M)$, then $(d\alpha)_\nu=d(\alpha_\nu)$.
\end{theorem}
\begin{proof}
Without loss of generality we may assume, rescaling the volume form, that $\nu(M)=1$.
Let us consider first the case of smooth functions. Let $f\in\cinf(M)$.
Then $f_\nu\in\R$ and $df_\nu=0$. On the other hand, for any $X\in\mathfrak{g}$, we have
\[
(df)_\nu(X)=\int_MX(f)\nu=\int_M\mathcal{L}_X(f)\nu=\int_M\mathcal{L}_X(f\nu),
\]
where the last equality follows since $\nu$ is right invariant and so it is invariant under the local flow of $X\in\mathfrak{g}$, which implies $\mathcal{L}_X\nu=0$. Now, by Cartan's formula, we derive
\[
\int_M\mathcal{L}_X(f\nu)=\int_Mi_X(d(f\nu))+\int_Md(i_X(f\nu))=0,
\]
where the first term in the above sum vanishes as $f\nu\in A^n(M)$ and so $d(f\nu)=0$, and the second term vanishes by the Stoke's theorem. Thus also $(df)_\nu=0$.

Now, recall the formula for the exterior derivative of $\alpha$. For any $X_0,\dots,X_k\in\mathfrak{g}$, we have
\begin{align*}
d\alpha(X_0,\dots,X_k)&=\sum_{1\le i\le n}(-1)^{i}X_i\big(\alpha(X_0,\dots,\hat{X_i},\dots,X_k)\big)+\\
&+\sum_{1\le i<j\le n}(-1)^{i+j}\alpha([X_i,X_j],X_0,\dots,\hat{X_i},\dots,\hat{X_j},\dots,X_k).
\end{align*}
Therefore,
\begin{align*}
d\alpha_\nu(X_0,\dots,X_k)&=\sum_{1\le i\le n}(-1)^{i}X_i\big(\int_M\alpha(X_0,\dots,\hat{X_i},\dots,X_k)\nu\big)+\\
&+\sum_{1\le i<j\le n}(-1)^{i+j}\int_M\alpha([X_i,X_j],X_0,\dots,\hat{X_i},\dots,\hat{X_j},\dots,X_k)\nu\\
&=\int_M\sum_{1\le i<j\le n}(-1)^{i+j}\alpha([X_i,X_j],X_0,\dots,\hat{X_i},\dots,\hat{X_j},\dots,X_k)\nu\\
&=\int_Md\alpha(X_0,\dots,X_k)\nu+\\
&-\int_M\sum_{1\le i\le n}(-1)^{i}X_i\big(\alpha(X_0,\dots,\hat{X_i},\dots,X_k)\big)\nu\\
&=\int_Md\alpha(X_0,\dots,X_k)\nu\\
&=(d\alpha)_\nu(X_0,\dots,X_k),
\end{align*}
where the penultimate equality follows from the case of smooth functions that we have already treated.
\end{proof}

Now, assume that $G$ admits a left invariant almost complex structure $J$, which descends to the quotient $M$. We observe that if $\psi\in A^{p,q}(M,J)$, then $\psi_\nu\in A^{p,q}(G,J)$ and it is left invariant.  We easily deduce the following lemma.

\begin{lemma}\label{lem-lie-group-almost-cplx}
Let $G$ be a $2n$-dimensional Lie group and let $\Gamma$ be a discrete subgroup such that $M=\Gamma\backslash G$ is compact. Let $\nu$ be a volume form on $M$ induced by a bi-invariant volume form on $G$. Assume that $J$ is a left invariant almost complex structure on $G$. Then for any $\psi\in A^{p,q}(M,J)$ we have
\[
(\mu\psi)_\nu=\mu\psi_\nu,\ \ \ (\del\psi)_\nu=\del\psi_\nu,\ \ \ (\delbar\psi)_\nu=\delbar\psi_\nu,\ \ \ (\mubar\psi)_\nu=\mubar\psi_\nu.
\]
\end{lemma}
\begin{proof}
The exterior derivative on $(p,q)$-forms decomposes into the direct sum $d=\mu+\del+\delbar+\mubar$, therefore
\begin{align*}
\mu\psi_\nu+\del\psi_\nu+\delbar\psi_\nu+\mubar\psi_\nu&=d\psi_\nu\\
&=(d\psi)_\nu\\
&=(\mu\psi+\del\psi+\delbar\psi+\mubar\psi)_\nu\\
&=(\mu\psi)_\nu+(\del\psi)_\nu+(\delbar\psi)_\nu+(\mubar\psi)_\nu,
\end{align*}
and the thesis follows because the symmetrization process behaves well with respect to the bidegree decomposition of forms.
\end{proof}

Let us now endow $(G,J)$ with a left invariant almost Hermitian metric $\omega$, which also descends to the quotient $M$. It turns out that the symmetrization process behaves well with the Hodge $*$ operator, here indicated by $*_\omega$.

\begin{lemma}\label{lem-lie-group-almost-herm}
Let $G$ be a $2n$-dimensional Lie group and let $\Gamma$ be a discrete subgroup such that $M=\Gamma\backslash G$ is compact. Let $\nu$ be a volume form on $M$ induced by a bi-invariant volume form on $G$. Assume that $(J,\omega)$ is a left invariant almost Hermitian structure on $G$. Then, for any $\psi\in A^{p,q}(M,J)$, we have
\[
(*_\omega\psi)_\nu=*_\omega\psi_\nu.
\]
\end{lemma}
\begin{proof}
Without loss of generality we may assume, rescaling the volume form, that $\nu(M)=1$.
Denote by $*_\omega:A^{p,q}(G,J)\to A^{n-q,n-p}(G,J)$ the Hodge $*$ operator on $(G,J,\omega)$ with respect to the volume form $\frac{\omega^n}{n!}$. Choose $\beta^1,\dots,\beta^n$ to be a global coframe of left invariant $(1,0)$-forms on $(G,J)$. %Let $V_1,\dots,V_n$ be the dual global frame of left invariant $(1,0)$-vector fields.
With the same notation as Section \ref{sec-preliminaries}, for any $\psi\in A^{p,q}(M,J)$, recall the formula \eqref{eq-local-hodge-star} for $*_\omega\psi$ with respect to the global coframe $\beta^1,\dots,\beta^n$:
\begin{equation*}%\label{eq-local-hodge-star}
*_\omega\psi=i^n(-1)^{\frac{n(n-1)}2+pq-n(p+q)}\det(g_{i\c{j}})\sum\epsilon_{A_p\c{B}_q}{\psi^{\c{A}_pB_q}}\beta^{{B}_n\setminus{B}_q \c{A}_n\setminus \c{A}_p}.
\end{equation*}
In this case, since $\omega$ is left invariant, note that $g_{i\c{j}}$ and $g^{\c{i}j}$ are complex constants. Therefore,
\begin{align*}
(*_\omega\psi)_\nu&=\int_M \left(i^n(-1)^{\frac{n(n-1)}2+pq-n(p+q)}\det(g_{i\c{j}})\sum\epsilon_{A_p\c{B}_q}{\psi^{\c{A}_pB_q}}\beta^{{B}_n\setminus{B}_q \c{A}_n\setminus \c{A}_p} \right) \nu\\
&=i^n(-1)^{\frac{n(n-1)}2+pq-n(p+q)}\det(g_{i\c{j}})\sum\epsilon_{A_p\c{B}_q}\left(\int_M{\psi^{\c{A}_pB_q}}\nu \right)\beta^{{B}_n\setminus{B}_q \c{A}_n\setminus \c{A}_p},
\end{align*}
and since
\begin{align*}
\psi_\nu&=\int_M \left( \sum\psi_{A_p\c{B}_q}\beta^{A_p\c{B}_q} \right) \nu\\
&=\sum\left(\int_M \psi_{A_p\c{B}_q} \nu\right)\beta^{A_p\c{B}_q},
\end{align*}
it also follows that
\begin{align*}
*_\omega\psi_\nu&=i^n(-1)^{\frac{n(n-1)}2+pq-n(p+q)}\det(g_{i\c{j}})\sum\epsilon_{A_p\c{B}_q}{\left(\psi_\nu\right)^{\c{A}_pB_q}}\beta^{{B}_n\setminus{B}_q \c{A}_n\setminus \c{A}_p}\\
&=i^n(-1)^{\frac{n(n-1)}2+pq-n(p+q)}\det(g_{i\c{j}})\sum\epsilon_{A_p\c{B}_q}\left(\int_M{\psi^{\c{A}_pB_q}}\nu \right)\beta^{{B}_n\setminus{B}_q \c{A}_n\setminus \c{A}_p}.
\end{align*} 
This ends the proof.
\end{proof}
We remark that, in order to prove the relation $(*_\omega\psi)_\nu=*_\omega\psi_\nu$, it is fundamental to assume that the almost Hermitian metric $\omega$ is left invariant on $G$, so that $g_{i\c{j}}$ and $g^{\c{i}j}$ are complex constants.

With an analogous proof, we note that Lemma \ref{lem-lie-group-almost-herm} can be generalized as follows.
\begin{proposition}
Let $G$ be a Lie group and let $\Gamma$ be a discrete subgroup such that $M=\Gamma\backslash G$ is compact. Let $\nu$ be a volume form on $M$ induced by a bi-invariant volume form on $G$. Let $g$ be a left invariant Riemannian metric on $G$. Then, for any $\psi\in A^k(M)$, we have
\[
(*_g\psi)_\nu=*_g\psi_\nu.
\]
\end{proposition}

We are now able to prove the following main theorem.
\begin{theorem}\label{thm-main}
Let $G$ be a $4$-dimensional Lie group and let $\Gamma$ be a discrete subgroup such that $M=\Gamma\backslash G$ is compact. Assume that $(J,\omega)$ is a left invariant almost Hermitian structure on $G$. Then, $h^{1,1}_\delbar(M,J,\omega)=b^-+1$ if and only if there exists a left invariant anti self dual $(1,1)$-form $\gamma\in A^{1,1}(G,J)$ satisfying
\[
id^c\gamma=d\omega.
\]
Otherwise, $h^{1,1}_\delbar(M,J,\omega)=b^-$.
\end{theorem}
\begin{proof}
First of all, recall that any left invariant almost Hermitian metric is Gauduchon. We recall here the proof for completeness. Indeed, if $\theta$ is its Lee form, then $d^{*_\omega}\theta$ is a left invariant function, i.e., a constant, and
\[
\int_Md^{*_\omega}\theta\frac{\omega^n}{n!}=0
\]
since $d^{*_\omega}$ is the $L^2$ formal adjoint of $d$, impying $d^{*_\omega}\theta=0$.

Therefore, if such a $\gamma$ exists, then the thesis follows from Theorem \ref{thm-dol-11-dim-char}.
Conversely, if $h^{1,1}_\delbar(M,J,\omega)=b^-+1$, then by Theorem \ref{thm-dol-11-dim-char} there exists an anti self dual $(1,1)$-form $\gamma\in A^{1,1}(M,J)$ satisfying $id^c\gamma=d\omega$. Let $\nu$ be a volume form on $M$ induced by a bi-invariant volume form on $G$ and let us consider the left invariant form $\gamma_\nu\in A^{1,1}(G,J)$. By Lemma \ref{lem-lie-group-almost-cplx}, we derive
\[
id^c\gamma_\nu=(id^c\gamma)_\nu=(d\omega)_\nu=d\omega,
\]
since $\omega$, and thus $d\omega$, are left invariant on $G$. Moreover, by Lemma \ref{lem-lie-group-almost-herm},
\[
*_\omega\gamma_\nu=(*_\omega\gamma)_\nu=(-\gamma)_\nu=-\gamma_\nu,
\]
that is, $\gamma_\nu$ is anti self dual. This concludes the proof.
\end{proof}

To apply Theorem \ref{thm-main} in explicit examples, we will need the following

\begin{lemma}\label{lem-gamma}
Let $G$ be a $4$-dimensional Lie group and let $\Gamma$ be a discrete subgroup such that $M=\Gamma\backslash G$ is compact. Assume that $(J,\omega)$ is a left invariant almost Hermitian structure on $G$.
Let $\phi^1,\phi^2$ be a global coframe of left invariant $(1,0)$-forms on $(G,J)$. The fundamental form $\omega$ is written as
\begin{equation}\label{eq-omega}
\omega=ir^2\phi^{1\c1}+is^2\phi^{2\c2}+u\phi^{1\c2}-\c{u}\phi^{2\c1},
\end{equation}
with $r,s>0$ and $u\in\C$ such that $rs>|u|$. Set $\tau=\sqrt{r^2s^2-|u|^2}\ne0$. Then, $\gamma$ is a left invariant anti self dual $(1,1)$-form on $(G,J)$ iff it can be written as
\begin{align}\label{eq-gamma-1}
\gamma&=Ar^2\phi^{1\c1}+\frac1{r^2}\big(A(2|u|^2-r^2s^2)+i\tau(B\c{u}-Cu)\big)\phi^{2\c2}+\\
&+\big(-iAu+B\tau\big)\phi^{1\c2}+\big(iA\c{u}+C\tau\big)\phi^{2\c1},\notag
\end{align}
for some $A,B,C\in\C$.
\end{lemma}
\begin{proof}
If we set
\begin{equation}\label{eq-psi-1}
\psi^1=r\phi^1+i\frac{\c{u}}{r}\phi^2,\ \ \ \psi^2=\frac{\sqrt{r^2s^2-|u|^2}}{r}\phi^2,
\end{equation}
then the fundamental form $\omega$ can be rewritten as
\[
\omega=i(\psi^{1\c1}+\psi^{2\c2}).
\]
It follows that $\gamma$ is a left invariant anti self dual $(1,1)$-form iff it is written as
\begin{equation}\label{eq-gamma-unitary}
\gamma=A\psi^{1\c1}+B\psi^{1\c2}+C\psi^{2\c1}-A\psi^{2\c2},
\end{equation}
with $A,B,C\in\C$. Then \eqref{eq-gamma-1} follows from plugging \eqref{eq-psi-1} into \eqref{eq-gamma-unitary}, since
\begin{align*}
\psi^{1\c1}&=r^2\phi^{1\c1}-iu\phi^{1\c2}+i\c{u}\phi^{2\c1}+\frac{|u|^2}{r^2}\phi^{2\c2},\\
\psi^{1\c2}&=\sqrt{r^2s^2-|u|^2}\big(\phi^{1\c2}+i\frac{\c{u}}{r^2}\phi^{2\c2}\big),\\
\psi^{2\c1}&=\sqrt{r^2s^2-|u|^2}\big(\phi^{2\c1}-i\frac{{u}}{r^2}\phi^{2\c2}\big),\\
\psi^{2\c2}&=\frac{r^2s^2-|u|^2}{r^2}\phi^{2\c2}.
\end{align*}
This ends the proof.
\end{proof}

Note that, varying $u\in\C$, $r,s\in\R$, $r,s>0$ and $rs>|u|$, any given left invariant almost Hermitian metric on $G$ can be written as
\begin{equation*}%\label{eq-omega}
\omega=ir^2\phi^{1\c1}+is^2\phi^{2\c2}+u\phi^{1\c2}-\c{u}\phi^{2\c1},
\end{equation*}
if $\phi^1,\phi^2$ is a global coframe of left invariant $(1,0)$-forms on $G$. Combining Theorem \ref{thm-main} and Lemma \ref{lem-gamma}, we obtain the following operative corollary.

\begin{corollary}\label{cor-char}
Let $G$ be a $4$-dimensional Lie group and let $\Gamma$ be a discrete subgroup such that $M=\Gamma\backslash G$ is compact. Assume that $(J,\omega)$ is a left invariant almost Hermitian structure on $G$.
Let $\phi^1,\phi^2$ be a global coframe of left invariant $(1,0)$-forms on $(G,J)$, and write $\omega$ as in \eqref{eq-omega}, with $u\in\C$, $r,s\in\R$, $r,s>0$ and $rs>|u|$. Set $\tau=\sqrt{r^2s^2-|u|^2}\ne0$. Then $h^{1,1}_\delbar(M,J,\omega)=b^-+1$ if and only if there exist $A,B,C\in\C$ such that
\begin{align*}%\label{eq-gamma-1}
&(ir^2-Ar^2)\del\phi^{1\c1}+\big(is^2-\frac1{r^2}\big(A(2|u|^2-r^2s^2)+i\tau(B\c{u}-Cu)\big)\big)\del\phi^{2\c2}+\\
&+\big(u+iAu-B\tau\big)\del\phi^{1\c2}+\big(-\c{u}-iA\c{u}-C\tau\big)\del\phi^{2\c1}=0,\notag
\end{align*}
and
\begin{align*}%\label{eq-gamma-1}
&(ir^2+Ar^2)\delbar\phi^{1\c1}+\big(is^2+\frac1{r^2}\big(A(2|u|^2-r^2s^2)+i\tau(B\c{u}-Cu)\big)\big)\delbar\phi^{2\c2}+\\
&+\big(u-iAu+B\tau\big)\delbar\phi^{1\c2}+\big(-\c{u}+iA\c{u}+C\tau\big)\delbar\phi^{2\c1}=0.\notag
\end{align*}
Otherwise, $h^{1,1}_\delbar(M,J,\omega)=b^-$.
\end{corollary}
\begin{proof}
By Theorem \ref{thm-main}, $h^{1,1}_\delbar(M,J,\omega)=b^-+1$ if and only if there exists a left invariant anti self dual $(1,1)$-form $\gamma$ satisfying $id^c\gamma=d\omega$. In general, the left invariant almost Hermitian metric $\omega$ is expressed by \eqref{eq-omega}, and a left invariant anti self dual $(1,1)$-form $\gamma$ on $G$ is written as in \eqref{eq-gamma-1} by Lemma \ref{lem-gamma}. Since $d^c=i(\delbar-\del)$ and $d=\del+\delbar$ on $(1,1)$-forms in real dimension $4$, then
\[
id^c\gamma=d\omega\quad \iff\quad \del(\omega-\gamma)=0,\quad \delbar(\omega+\gamma)=0.
\]
Combining this together with \eqref{eq-omega} and \eqref{eq-gamma-1}, we derive the thesis.
\end{proof}
We remark that $h^{1,1}_\delbar(M,J,\omega)=b^-+1$ in Corollary \ref{cor-char} does not depend on the choice of the global coframe $\phi^1,\phi^2$ of left invariant $(1,0)$-forms on $G$.

With a similar proof to Theorem \ref{thm-main}, using Theorem \ref{thm-belgun}, we can also prove that there exists an almost K\"ahler metric on $M$ if and only if there exists a left invariant almost K\"ahler metric on $G$.

\begin{proposition}\label{prop-alm-kahl}
Let $G$ be a $2n$-dimensional Lie group and let $\Gamma$ be a discrete subgroup such that $M=\Gamma\backslash G$ is compact. Assume that $J$ is a left invariant almost complex structure on $G$. Then, there exists an almost K\"ahler metric on $(M,J)$ if and only if there exists a left invariant almost K\"ahler metric on $(G,J)$.
\end{proposition}
\begin{proof}
The only thing to prove is that if there exists an almost K\"ahler metric $\omega$ on $M$ then there exists a left invariant almost K\"ahler metric $\tilde\omega$ on $G$. Let $\nu$ be a volume form on $M$ induced by a bi-invariant volume form on $G$. Choose $\tilde\omega=\omega_\nu$. Note that $\omega_\nu$ is a positive real $(1,1)$-form and, since $d\omega=0$, then $d\omega_\nu=(d\omega)_\nu=0$ by Theorem \ref{thm-belgun}.
\end{proof}

\section{Applications}\label{sec-applications}
This section is devoted to the study of the number $h^{1,1}_\delbar$ on explicit examples of almost Hermitian $4$-manifolds which are obtained as the compact quotient of a Lie group by a discrete subgroup, through the systematic use of Corollary \ref{cor-char}. Motivated by Question \ref{question-zhang}, we also check if there exist almost K\"ahler metrics on such manifolds through the use of Proposition \ref{prop-alm-kahl}.

\subsection{Secondary Kodaira surface}
Let $M=\Gamma\backslash G$ be a secondary Kodaira surface. Here $G$ is a solvable Lie group and $\Gamma$ is a cocompact lattice. We refer to \cite[pp. 756, 760]{Ha} for the construction of $M$ and for the structure equations of the global coframe $\{e^1,e^2,e^3,e^4\}$ of left invariant $1$-forms on $G$,
\[
de^1=e^{24},\ \ \ de^2=-e^{14},\ \ \ de^3=e^{12},\ \ \ de^4=0.
\]
Recall that $b^2=0$, therefore in particular $b^-=0$.

We endow $G$ and $M$ with the left invariant almost complex structure $J$ given by
\[
\phi^1=e^1+ie^3,\ \ \ \phi^2=e^2+ie^4
\]
being a global coframe of the vector bundle of $(1,0)$ forms $T^{1,0}G$. The associated structure equations are
\begin{align*}
&d\phi^1=\frac{i}4\big(\phi^{12}+\phi^{1\c2}-\phi^{2\c1}+2\phi^{2\c2}+\phi^{\c1\c2}\big),\\
&d\phi^2=\frac{i}4\big(\phi^{12}-\phi^{1\c2}-\phi^{2\c1}-\phi^{\c1\c2}\big),
\end{align*}
therefore the almost complex structure $J$ is non integrable. This is the same almost complex structure considered in \cite[Section 6]{piovani-parma}.
From the structure equations we derive
\begin{align*}
4i\del\phi^{1\c1}&=2\phi^{12\c2},&
4i\delbar\phi^{1\c1}&=2\phi^{2\c1\c2},\\
4i\del\phi^{1\c2}&=-\phi^{12\c1}-\phi^{12\c2},&
4i\delbar\phi^{1\c2}&=-\phi^{1\c1\c2}+\phi^{2\c1\c2},\\
4i\del\phi^{2\c1}&=-\phi^{12\c1}+\phi^{12\c2},&
4i\delbar\phi^{2\c1}&=-\phi^{1\c1\c2}-\phi^{2\c1\c2},\\
4i\del\phi^{2\c2}&=0,&
4i\delbar\phi^{2\c2}&=0.
\end{align*}

Endow $(G,J)$ with a left invariant almost Hermitian metric
\begin{equation*}%\label{eq-omega}
\omega=ir^2\phi^{1\c1}+is^2\phi^{2\c2}+u\phi^{1\c2}-\c{u}\phi^{2\c1},
\end{equation*}
with $u\in\C$, $r,s\in\R$, $r,s>0$ and $rs>|u|$. Set $\tau=\sqrt{r^2s^2-|u|^2}\ne0$.

By Corollary \ref{cor-char}, we have that $h^{1,1}_\delbar(M,J,\omega)=b^-+1$ if and only if there exist $A,B,C\in\C$ such that
\begin{align*}%\label{eq-gamma-1}
(ir^2-Ar^2)2\phi^{12\c2}+&(u+iAu-B\tau)\big(-\phi^{12\c1}-\phi^{12\c2}\big)+\\
+&(-\c{u}-iA\c{u}-C\tau)\big(-\phi^{12\c1}+\phi^{12\c2}\big)=0,
\end{align*}
and
\begin{align*}
(ir^2+Ar^2)2\phi^{2\c1\c2}+&(u-iAu+B\tau)\big(-\phi^{1\c1\c2}+\phi^{2\c1\c2}\big)+\\
+&(-\c{u}+iA\c{u}+C\tau)\big(-\phi^{1\c1\c2}-\phi^{2\c1\c2}\big)=0,
\end{align*}
if and only if there exist $A,B,C\in\C$ such that
\begin{numcases}{}
-u-iAu+B\tau+\c{u}+iA\c{u}+C\tau=0,\label{eq-sec-kod-1}\\
-u+iAu-B\tau+\c{u}-iA\c{u}-C\tau=0,\label{eq-sec-kod-2}\\
2ir^2-2Ar^2-u-iAu+B\tau-\c{u}-iA\c{u}-C\tau=0,\label{eq-sec-kod-3}\\
2ir^2+2Ar^2+u-iAu+B\tau+\c{u}-iA\c{u}-C\tau=0.\label{eq-sec-kod-4}
\end{numcases}
Summing \eqref{eq-sec-kod-1} and \eqref{eq-sec-kod-2}, subtracting \eqref{eq-sec-kod-1} from \eqref{eq-sec-kod-2}, summing \eqref{eq-sec-kod-3} and \eqref{eq-sec-kod-4}, subtracting \eqref{eq-sec-kod-3} from \eqref{eq-sec-kod-4}, we obtain
\begin{equation*}
\begin{cases}
u=\c{u},\\
B=-C,\\
C=\frac{i(r^4+u^2)}{r^2\sqrt{r^2s^2-u^2}},\\
A=-\frac{u}{r^2}.
\end{cases}
\end{equation*}
Therefore, $\img(u)=0$ if and only if $h^{1,1}_\delbar(M,J,\omega)=b^-+1=1$, and $\img(u)\ne0$ if and only if $h^{1,1}_\delbar(M,J,\omega)=b^-=0$.

Note that, since $b^2=0$, then there exist no symplectic forms on $M$,  and thus no almost K\"ahler metrics on $(M,J)$. This can be reproved with the help of Proposition \ref{prop-alm-kahl}. Indeed, there exists an almost K\"ahler metric on $(M,J)$ if and only if there exists a left invariant almost K\"ahler metric on $(G,J)$. Since $d\omega=\del\omega+\delbar\omega$ and $\c\omega=\omega$, then $d\omega=0$ if and only if $\delbar\omega=0$. We have $\delbar\omega=0$ if and only if
\begin{equation*}
\begin{cases}
u=\c{u},\\
u=-ir^2,
\end{cases}
\end{equation*}
which implies $u=r=0$, but this cannot happen since $r>0$.

Summing up, we have just proved the following
\begin{proposition}\label{prop-sec-kod}
Let $M=\Gamma\backslash G$, $J$, $\omega$ be as above. We have $h^{1,1}_\delbar(M,J,\omega)=b^-+1=1$ iff $\img(u)=0$, otherwise $h^{1,1}_\delbar(M,J,\omega)=b^-=0$. Moreover, there exist no almost K\"ahler metrics on $(M,J)$.
\end{proposition}

This gives a negative answer to Question \ref{question-zhang}.

\subsection{Inoue surface $S_M$}

Let $M=\Gamma\backslash G$ be a Inoue surface of type $S_M$. Here $G$ is a solvable Lie group and $\Gamma$ is a cocompact lattice. We refer to \cite[pp. 755, 760]{Ha} for its construction and for the structure equations of the global coframe $\{e^1,e^2,e^3,e^4\}$ of left invariant $1$-forms on $G$. For any $\alpha,\beta\in\R$, $\alpha\ne0$,
\[
de^1=\alpha e^{14}+\beta e^{24},\ \ \ de^2=-\beta e^{14}+\alpha e^{24},\ \ \ de^3=-2\alpha e^{34},\ \ \ de^4=0.
\]
Recall that $b^2=0$, therefore in particular $b^-=0$.

We endow $G$ and $M$ with the left invariant almost complex structure $J$ given by
\[
\phi^1=e^1+ie^3,\ \ \ \phi^2=e^2+ie^4
\]
being a global coframe of the vector bundle of $(1,0)$ forms $T^{1,0}G$. The associated structure equations are
\begin{align*}
&d\phi^1=\alpha\frac{i}4\big(\phi^{12}-\phi^{1\c2}+3\phi^{2\c1}+3\phi^{\c1\c2}\big)+\beta\frac{i}2\phi^{2\c2},\\
&d\phi^2=\beta\frac{i}4\big(\phi^{12}-\phi^{1\c2}-\phi^{2\c1}-\phi^{\c1\c2}\big)+\alpha\frac{i}2\phi^{2\c2},
\end{align*}
therefore the almost complex structure $J$ is non integrable.
This is the same almost complex structure considered in \cite[Section 6]{piovani-parma}.
From the structure equations we derive
\begin{align*}
4i\del\phi^{1\c1}&=-2\alpha\phi^{12\c1}+2\beta\phi^{12\c2},&
4i\delbar\phi^{1\c1}&=-2\alpha\phi^{1\c1\c2}+2\beta\phi^{2\c1\c2},\\
4i\del\phi^{1\c2}&=-\beta\phi^{12\c1}+\alpha\phi^{12\c2},&
4i\delbar\phi^{1\c2}&=-\beta\phi^{1\c1\c2}-3\alpha\phi^{2\c1\c2},\\
4i\del\phi^{2\c1}&=-\beta\phi^{12\c1}-3\alpha\phi^{12\c2},&
4i\delbar\phi^{2\c1}&=-\beta\phi^{1\c1\c2}+\alpha\phi^{2\c1\c2},\\
4i\del\phi^{2\c2}&=0,&
4i\delbar\phi^{2\c2}&=0.
\end{align*}

Endow $(G,J)$ with a left invariant almost Hermitian metric
\begin{equation*}%\label{eq-omega}
\omega=ir^2\phi^{1\c1}+is^2\phi^{2\c2}+u\phi^{1\c2}-\c{u}\phi^{2\c1},
\end{equation*}
with $u\in\C$, $r,s\in\R$, $r,s>0$ and $rs>|u|$. Here $r,s,u$ may depend on $\alpha,\beta$. Set $\tau=\sqrt{r^2s^2-|u|^2}\ne0$.

By Corollary \ref{cor-char}, we have that $h^{1,1}_\delbar(M,J,\omega)=b^-+1$ if and only if there exist $A,B,C\in\C$ such that
\begin{align*}%\label{eq-gamma-1}
(ir^2-Ar^2)\big(-2\alpha\phi^{12\c1}+2\beta\phi^{12\c2}\big)+&(u+iAu-B\tau)\big(-\beta\phi^{12\c1}+\alpha\phi^{12\c2}\big)+\\
+&(-\c{u}-iA\c{u}-C\tau)\big(-\beta\phi^{12\c1}-3\alpha\phi^{12\c2}\big)=0,
\end{align*}
and
\begin{align*}
(ir^2+Ar^2)\big(-2\alpha\phi^{1\c1\c2}+2\beta\phi^{2\c1\c2}\big)+&(u-iAu+B\tau)\big(-\beta\phi^{1\c1\c2}-3\alpha\phi^{2\c1\c2}\big)+\\
+&(-\c{u}+iA\c{u}+C\tau)\big(-\beta\phi^{1\c1\c2}+\alpha\phi^{2\c1\c2}\big)=0,
\end{align*}
if and only if there exist $A,B,C\in\C$ such that
\begin{numcases}{}
-2\alpha(ir^2-Ar^2)-\beta(u+iAu-B\tau)+\beta(\c{u}+iA\c{u}+C\tau)=0,\label{eq-ino-sm-1}\\
-2\alpha(ir^2+Ar^2)-\beta(u-iAu+B\tau)+\beta(\c{u}-iA\c{u}-C\tau)=0,\label{eq-ino-sm-2}\\
2\beta(ir^2-Ar^2)+\alpha(u+iAu-B\tau)+3\alpha(\c{u}+iA\c{u}+C\tau)=0,\label{eq-ino-sm-3}\\
2\beta(ir^2+Ar^2)-3\alpha(u-iAu+B\tau)+\alpha(-\c{u}+iA\c{u}+C\tau)=0.\label{eq-ino-sm-4}
\end{numcases}
Summing \eqref{eq-ino-sm-1} and \eqref{eq-ino-sm-2}, subtracting \eqref{eq-ino-sm-1} from \eqref{eq-ino-sm-2}, summing \eqref{eq-ino-sm-3} and \eqref{eq-ino-sm-4}, subtracting \eqref{eq-ino-sm-3} from \eqref{eq-ino-sm-4}, we obtain
\begin{equation*}
\begin{cases}
\beta(u-\c{u})=-2i\alpha r^2,\\
B=-C,\\
C=\frac{3(3u+\c{u})\alpha^3-6ir^2\alpha^2\beta+(3u+\c{u})\alpha\beta^2-2ir^2\beta^3}{4\alpha(\alpha^2+\beta^2)\sqrt{r^2s^2-|u|^2}},\\
A=-\frac{2i\alpha^2}{\alpha^2+\beta^2}.
\end{cases}
\end{equation*}
Therefore, $\beta\img(u)=-\alpha r^2$ if and only if $h^{1,1}_\delbar(M,J,\omega)=b^-+1=1$, and $\beta\img(u)\ne-\alpha r^2$ if and only if $h^{1,1}_\delbar(M,J,\omega)=b^-=0$.

Note that, since $b^2=0$, then there exist no symplectic forms on $M$,  and thus no almost K\"ahler metrics on $(M,J)$.
Let us also check it explicitly. 
We have $\delbar\omega=0$ if and only if
\begin{equation*}
\begin{cases}
\beta(u-\c{u})=-2i\alpha r^2,\\
u(3\alpha^2+\beta^2)=\c{u}(\beta^2-\alpha^2),
\end{cases}
\end{equation*}
which gives a contradiction. Therefore, by Proposition \ref{prop-alm-kahl}, there exist no almost K\"ahler metrics on $(M,J)$.

Summing up, we have just proved the following
\begin{proposition}\label{prop-inoue}
Let $M=\Gamma\backslash G$, $J$, $\omega$ be as above. We have $h^{1,1}_\delbar(M,J,\omega)=b^-+1=1$ iff $\beta\img(u)=-\alpha r^2$, otherwise $h^{1,1}_\delbar(M,J,\omega)=b^-=0$.
In particular, if $\beta=0$, then $h^{1,1}_\delbar(M,J,\omega)=b^-=0$ for any left invariant almost Hermitian metric $\omega$ on $G$.
Moreover, there exist no almost K\"ahler metrics on $M$.
\end{proposition}

This gives a negative answer to Question \ref{question-zhang}.

\subsection{$4$-nilmanifold without complex structures, I}

Let $M=\Gamma\backslash G$ be the $4$-dimensional nilmanifold which does not admit integrable almost complex structures. Here $G$ is a nilpotent Lie group and $\Gamma$ is a cocompact lattice. We refer to \cite[p. 758]{Ha} for its construction and to \cite[p. 7]{cattaneo-nannicini-tomassini} for the structure equations of the global coframe $\{e^1,e^2,e^3,e^4\}$ of left invariant $1$-forms on $G$,
\[
de^1=0,\ \ \ de^2=0,\ \ \ de^3=- e^{12},\ \ \ de^4=-e^{13}.
\]
Recall that $b^2=2$ and $b^-=1$.

We endow $G$ and $M$ with the left invariant almost complex structure $J$ given by
\[
\phi^1=e^3+ie^4,\ \ \ \phi^2=e^1+ie^2
\]
being a global coframe of the vector bundle of $(1,0)$ forms $T^{1,0}G$. If we switch $\phi^1$ and $\phi^2$, then this is the same almost complex structure considered in \cite[Section 3]{PT5} and  in \cite[p. 20]{cattaneo-nannicini-tomassini}. The associated structure equations are
\begin{align*}
&d\phi^1=\frac{i}4\big(\phi^{12}+\phi^{1\c2}-\phi^{2\c1}-2\phi^{2\c2}+\phi^{\c1\c2}\big),\\
&d\phi^2=0.
\end{align*}
Note that the almost complex structure $J$ is non integrable.
From the structure equations we derive
\begin{align*}
4i\del\phi^{1\c1}&=-2\phi^{12\c2},&
4i\delbar\phi^{1\c1}&=-2\phi^{2\c1\c2},\\
4i\del\phi^{1\c2}&=-\phi^{12\c2},&
4i\delbar\phi^{1\c2}&=\phi^{2\c1\c2},\\
4i\del\phi^{2\c1}&=\phi^{12\c2},&
4i\delbar\phi^{2\c1}&=-\phi^{2\c1\c2},\\
4i\del\phi^{2\c2}&=0,&
4i\delbar\phi^{2\c2}&=0.
\end{align*}

Endow $(G,J)$ with a left invariant almost Hermitian metric
\begin{equation*}%\label{eq-omega}
\omega=ir^2\phi^{1\c1}+is^2\phi^{2\c2}+u\phi^{1\c2}-\c{u}\phi^{2\c1},
\end{equation*}
with $u\in\C$, $r,s\in\R$, $r,s>0$ and $rs>|u|$. Set $\tau=\sqrt{r^2s^2-|u|^2}\ne0$.

By Corollary \ref{cor-char}, we have that $h^{1,1}_\delbar(M,J,\omega)=b^-+1$ if and only if there exist $A,B,C\in\C$ such that
\begin{numcases}{}
-2ir^2+2Ar^2-u-iAu+B\tau-\c{u}-iA\c{u}-C\tau=0,\label{eq-nil1-1}\\
-2ir^2-2Ar^2+u-iAu+B\tau+\c{u}-iA\c{u}-C\tau=0.\label{eq-nil1-2}
\end{numcases}
Summing \eqref{eq-nil1-1} and \eqref{eq-nil1-2}, subtracting \eqref{eq-nil1-1} from \eqref{eq-nil1-2}, we obtain
\begin{equation*}
\begin{cases}
A=r^2\real(u),\\
B-C=\frac{2ir^2(1+\real(u)^2r^4)}{\sqrt{r^2s^2-|u|^2}}.
\end{cases}
\end{equation*}
Therefore it always holds that $h^{1,1}_\delbar(M,J,\omega)=b^-+1=2$.

Note that $\omega$ is almost K\"ahler, i.e., $\delbar\omega=0$, if and only if
\begin{equation*}
\real(u)=ir^2.
\end{equation*}

Summing up, we have just proved the following
\begin{proposition}
Let $M=\Gamma\backslash G$, $J$, $\omega$ be as above. We have $h^{1,1}_\delbar(M,J,\omega)=b^-+1=2$ for every left invariant almost Hermitian metric $\omega$ on $(G,J)$.
Moreover, $\omega$ is almost K\"ahler iff $\real(u)=ir^2$.
\end{proposition}

\subsection{$4$-nilmanifold without complex structures, II}

Let $M=\Gamma\backslash G$ be the $4$-dimensional nilmanifold which does not admit integrable almost complex structures as in the previous subsection.

We endow $G$ and $M$ with the left invariant almost complex structure $J$ given by
\[
\phi^1=e^1+ie^4,\ \ \ \phi^2=e^2+ie^3
\]
being a global coframe of the vector bundle of $(1,0)$ forms $T^{1,0}G$.  The associated structure equations are
\begin{align*}
&d\phi^1=\frac{1}4\big(-\phi^{12}+\phi^{1\c2}+\phi^{2\c1}+\phi^{\c1\c2}\big),\\
&d\phi^2=-\frac{i}4\big(\phi^{12}+\phi^{1\c2}-\phi^{2\c1}+\phi^{\c1\c2}\big).
\end{align*}
Note that the almost complex structure $J$ is non integrable.
This is the same almost complex structure considered in \cite[p. 18]{cattaneo-nannicini-tomassini}.
From the structure equations we derive
\begin{align*}
4i\del\phi^{1\c1}&=0,&
4i\delbar\phi^{1\c1}&=0,\\
4i\del\phi^{1\c2}&=-\phi^{12\c1}-i\phi^{12\c2},&
4i\delbar\phi^{1\c2}&=\phi^{1\c1\c2}+i\phi^{2\c1\c2},\\
4i\del\phi^{2\c1}&=\phi^{12\c1}-i\phi^{12\c2},&
4i\delbar\phi^{2\c1}&=-\phi^{1\c1\c2}+i\phi^{2\c1\c2},\\
4i\del\phi^{2\c2}&=0,&
4i\delbar\phi^{2\c2}&=0.
\end{align*}

Endow $(G,J)$ with a left invariant almost Hermitian metric
\begin{equation*}%\label{eq-omega}
\omega=ir^2\phi^{1\c1}+is^2\phi^{2\c2}+u\phi^{1\c2}-\c{u}\phi^{2\c1},
\end{equation*}
with $u\in\C$, $r,s\in\R$, $r,s>0$ and $rs>|u|$. Set $\tau=\sqrt{r^2s^2-|u|^2}\ne0$.

By Corollary \ref{cor-char}, we have that $h^{1,1}_\delbar(M,J,\omega)=b^-+1$ if and only if there exist $A,B,C\in\C$ such that
\begin{numcases}{}
-u-iAu+B\tau-\c{u}-iA\c{u}-C\tau=0,\label{eq-nil2-1}\\
u+iAu-B\tau-\c{u}-iA\c{u}-C\tau=0,\label{eq-nil2-2}\\
u-iAu+B\tau+\c{u}-iA\c{u}-C\tau=0,\label{eq-nil2-3}\\
u-iAu+B\tau-\c{u}+iA\c{u}+C\tau=0.\label{eq-nil2-4}
\end{numcases}
Summing \eqref{eq-nil2-1} and \eqref{eq-nil2-2}, subtracting \eqref{eq-nil2-1} from \eqref{eq-nil2-2}, summing \eqref{eq-nil2-3} and \eqref{eq-nil2-4}, subtracting \eqref{eq-nil2-3} from \eqref{eq-nil2-4}, we obtain
\[
B=C=u=0.
\]
Therefore, $u=0$ if and only if $h^{1,1}_\delbar(M,J,\omega)=b^-+1=2$, and $u\ne0$ if and only if $h^{1,1}_\delbar(M,J,\omega)=b^-=1$.

Note that $\omega$ is almost K\"ahler, i.e., $\delbar\omega=0$, if and only if
\begin{equation*}
u=0.
\end{equation*}

Summing up, we have just proved the following
\begin{proposition}
Let $M=\Gamma\backslash G$, $J$, $\omega$ be as above. We have $h^{1,1}_\delbar(M,J,\omega)=b^-+1=2$ iff $u=0$, otherwise $h^{1,1}_\delbar(M,J,\omega)=b^-=1$.
Moreover, $\omega$ is almost K\"ahler iff $u=0$.
\end{proposition}

\subsection{Hyperelliptic surface, I}

Let $M=\Gamma\backslash G$ be a Hyperelliptic surface. Here $G$ is a solvable Lie group and $\Gamma$ is a cocompact lattice. We refer to \cite[pp. 754,760]{Ha} for its construction and for the structure equations of the global coframe $\{e^1,e^2,e^3,e^4\}$ of left invariant $1$-forms on $G$,
\[
de^1=-e^{23},\ \ \ de^2=e^{13},\ \ \ de^3=0,\ \ \ de^4=0.
\]
Recall that $b^2=2$ and $b^-=1$.

We endow $G$ and $M$ with the left invariant almost complex structure $J$ given by
\[
\phi^1=e^1+ie^3,\ \ \ \phi^2=e^2+ie^4
\]
being a global coframe of the vector bundle of $(1,0)$ forms $T^{1,0}G$.  The associated structure equations are
\begin{align*}
&d\phi^1=-\frac{i}4\big(\phi^{12}+\phi^{1\c2}+\phi^{2\c1}-\phi^{\c1\c2}\big),\\
&d\phi^2=\frac{i}2\phi^{1\c1},
\end{align*}
therefore the almost complex structure $J$ is non integrable.
This is the same almost complex structure considered in \cite[Section 6]{PT4}.
From the structure equations we derive
\begin{align*}
4i\del\phi^{1\c1}&=0,&
4i\delbar\phi^{1\c1}&=0,\\
4i\del\phi^{1\c2}&=\phi^{12\c2},&
4i\delbar\phi^{1\c2}&=\phi^{2\c1\c2},\\
4i\del\phi^{2\c1}&=\phi^{12\c2},&
4i\delbar\phi^{2\c1}&=\phi^{2\c1\c2},\\
4i\del\phi^{2\c2}&=-2\phi^{12\c1},&
4i\delbar\phi^{2\c2}&=-2\phi^{1\c1\c2}.
\end{align*}

Endow $(G,J)$ with a left invariant almost Hermitian metric
\begin{equation*}%\label{eq-omega}
\omega=ir^2\phi^{1\c1}+is^2\phi^{2\c2}+u\phi^{1\c2}-\c{u}\phi^{2\c1},
\end{equation*}
with $u\in\C$, $r,s\in\R$, $r,s>0$ and $rs>|u|$. Set $\tau=\sqrt{r^2s^2-|u|^2}\ne0$.

By Corollary \ref{cor-char}, we have that $h^{1,1}_\delbar(M,J,\omega)=b^-+1$ if and only if there exist $A,B,C\in\C$ such that
\begin{numcases}{}
is^2r^2-A(2|u|^2-r^2s^2)-i\tau(B\c{u}-Cu)=0,\label{eq-hyp1-1}\\
is^2r^2+A(2|u|^2-r^2s^2)+i\tau(B\c{u}-Cu)=0,\label{eq-hyp1-2}\\
u+iAu-B\tau-\c{u}-iA\c{u}-C\tau=0,\label{eq-hyp1-3}\\
u-iAu+B\tau-\c{u}+iA\c{u}+C\tau=0.\label{eq-hyp1-4}
\end{numcases}
Summing \eqref{eq-hyp1-1} and \eqref{eq-hyp1-2}, we obtain
\[
is^2r^2=0,
\]
which cannot be since $r,s>0$.
Therefore it always holds that $h^{1,1}_\delbar(M,J,\omega)=b^-=1$.

Let us also check if there exist almost K\"ahler metrics on $(M,J)$.
Note that $\omega$ is almost K\"ahler, i.e., $\delbar\omega=0$, if and only if
\begin{equation*}
\begin{cases}
u=\c{u},\\
is^2=0,
\end{cases}
\end{equation*}
which cannot be since $s>0$.
Therefore, by Proposition \ref{prop-alm-kahl}, there exist no almost K\"ahler metrics on $(M,J)$.

Summing up, we have just proved the following
\begin{proposition}
Let $M=\Gamma\backslash G$, $J$, $\omega$ be as above. We have $h^{1,1}_\delbar(M,J,\omega)=b^-=1$ for every left invariant almost Hermitian metric $\omega$ on $(G,J)$.
Moreover, there exist no almost K\"ahler metrics on $(M,J)$.
\end{proposition}

\subsection{Hyperelliptic surface, II}

Let $M=\Gamma\backslash G$ be a Hyperelliptic surface, as in the previous subsection.

We endow $G$ and $M$ with the deformed left invariant almost complex structure $J_t$ given by
\[
\phi^1=(1+t)e^1+i(1-t)e^2,\ \ \ \phi^2=e^3+ie^4,
\]
with $t\in\C$, $0<|t|<1$,
being a global coframe of the vector bundle of $(1,0)$ forms $T^{1,0}G$. 
This is the same deformed almost complex structure considered in \cite[Section 2]{PT5}. For $t=0$, this is the usual K\"ahler complex structure on Hyperelliptic surfaces.
The associated structure equations are
\begin{align*}
&d\phi^1=\frac{i}{2(1-|t|^2)}\big((1+|t|^2)(\phi^{12}+\phi^{1\c2})+2t(\phi^{2\c1}-\phi^{\c1\c2})\big),\\
&d\phi^2=0,
\end{align*}
therefore the almost complex structure $J_t$ is non integrable.
From the structure equations we derive
\begin{align*}
2i(1-|t|^2)\del\phi^{1\c1}&=0,&
2i(1-|t|^2)\delbar\phi^{1\c1}&=0,\\
2i(1-|t|^2)\del\phi^{1\c2}&=-(1+|t|^2)\phi^{12\c2},&
2i(1-|t|^2)\delbar\phi^{1\c2}&=-2t\phi^{2\c1\c2},\\
2i(1-|t|^2)\del\phi^{2\c1}&=-2\c{t}\phi^{12\c2},&
2i(1-|t|^2)\delbar\phi^{2\c1}&=-(1+|t|^2)\phi^{2\c1\c2},\\
2i(1-|t|^2)\del\phi^{2\c2}&=0,&
2i(1-|t|^2)\delbar\phi^{2\c2}&=0.
\end{align*}

Endow $(G,J_t)$ with a left invariant almost Hermitian metric
\begin{equation*}%\label{eq-omega}
\omega_t=ir^2\phi^{1\c1}+is^2\phi^{2\c2}+u\phi^{1\c2}-\c{u}\phi^{2\c1},
\end{equation*}
with $u\in\C$, $r,s\in\R$, $r,s>0$ and $rs>|u|$. Here $r,s,u$ may depend on $t$. Set $\tau=\sqrt{r^2s^2-|u|^2}\ne0$.

By Corollary \ref{cor-char}, we have that $h^{1,1}_\delbar(M,J_t,\omega_t)=b^-+1$ if and only if there exist $A,B,C\in\C$ such that
\begin{equation*}
\begin{cases}
(1+|t|^2)(u+iAu-B\tau)+2\c{t}(-\c{u}-iA\c{u}-C\tau)=0,\\
2t(u-iAu+B\tau)+(1+|t|^2)(-\c{u}+iA\c{u}+C\tau)=0.
\end{cases}
\end{equation*}
It is easy to check that the rank of the matrix associated to this system is $2$.
Therefore it always holds that $h^{1,1}_\delbar(M,J_t,\omega_t)=b^-+1=2$.

%Let us also check if there exist almost K\"ahler metrics on $(M,J)$.
Note that $\omega_t$ is almost K\"ahler, i.e., $\delbar\omega_t=0$, if and only if
\begin{equation*}
2tu=(1+|t|^2)\c{u},
\end{equation*}
which is equivalent to $u=0$.
%Therefore, by Proposition \ref{prop-alm-kahl}, there exist no almost K\"ahler metrics on $(M,J)$.

Summing up, we have just proved the following
\begin{proposition}
Let $M=\Gamma\backslash G$, $J_t$, $\omega_t$ be as above. We have $h^{1,1}_\delbar(M,J_t,\omega_t)=b^-+1=2$ for every left invariant almost Hermitian metric $\omega_t$ on $(G,J_t)$.
Moreover, $\omega_t$ is almost K\"ahler iff $u=0$.
\end{proposition}

\subsection{Primary Kodaira surface, I}

Let $M=\Gamma\backslash G$ be the primary Kodaira surface. Here $G$ is a nilpotent Lie group and $\Gamma$ is a cocompact lattice. We refer to \cite[pp. 755,760]{Ha} for its construction and for the structure equations of the global coframe $\{e^1,e^2,e^3,e^4\}$ of left invariant $1$-forms on $G$,
\[
de^1=0,\ \ \ de^2=0,\ \ \ de^3=-e^{12},\ \ \ de^4=0.
\]
Recall that $b^2=2$ and $b^-=1$.

We endow $G$ and $M$ with the left invariant almost complex structure $J_\alpha$ given by
\[
\phi^1=(e^1+\alpha e^4)+ie^3,\ \ \ \phi^2=e^2+ie^4,
\]
with $\alpha\in\R$,
being a global coframe of the vector bundle of $(1,0)$ forms $T^{1,0}G$.  The associated structure equations are
\begin{align*}
&d\phi^1=-\frac{1}4\big(i\phi^{12}+i\phi^{1\c2}-i\phi^{2\c1}-2\alpha\phi^{2\c2}+i\phi^{\c1\c2}\big),\\
&d\phi^2=0,
\end{align*}
therefore the almost complex structure $J_\alpha$ is non integrable.
This is the same almost complex structure considered in \cite[Example 4.1]{TT}.
From the structure equations we derive
\begin{align*}
4\del\phi^{1\c1}&=2\alpha\phi^{12\c2},&
4\delbar\phi^{1\c1}&=-2\alpha\phi^{2\c1\c2},\\
4\del\phi^{1\c2}&=-i\phi^{12\c2},&
4\delbar\phi^{1\c2}&=i\phi^{2\c1\c2},\\
4\del\phi^{2\c1}&=i\phi^{12\c2},&
4\delbar\phi^{2\c1}&=-i\phi^{2\c1\c2},\\
4\del\phi^{2\c2}&=0,&
4\delbar\phi^{2\c2}&=0.
\end{align*}

Endow $(G,J_\alpha)$ with a left invariant almost Hermitian metric
\begin{equation*}%\label{eq-omega}
\omega_\alpha=ir^2\phi^{1\c1}+is^2\phi^{2\c2}+u\phi^{1\c2}-\c{u}\phi^{2\c1},
\end{equation*}
with $u\in\C$, $r,s\in\R$, $r,s>0$ and $rs>|u|$. Here $r,s,u$ may depend on $\alpha$. Set $\tau=\sqrt{r^2s^2-|u|^2}\ne0$.

By Corollary \ref{cor-char}, we have that $h^{1,1}_\delbar(M,J_\alpha,\omega_\alpha)=b^-+1$ if and only if there exist $A,B,C\in\C$ such that
\begin{numcases}{}
2i\alpha r^2-2A\alpha r^2-iu+Au+iB\tau-i\c{u}+A\c{u}-iC\tau=0,\label{eq-kod1-1}\\
-2i\alpha r^2-2A\alpha r^2+iu+Au+iB\tau+i\c{u}+A\c{u}-iC\tau=0.\label{eq-kod1-2}
\end{numcases}
Summing \eqref{eq-kod1-1} and \eqref{eq-kod1-2}, subtracting \eqref{eq-kod1-1} from \eqref{eq-kod1-2}, we obtain
\begin{equation*}
\begin{cases}
u+\c{u}=2\alpha r^2,\\
B=C.
\end{cases}
\end{equation*}
Therefore, $\real(u)=\alpha r^2$ if and only if $h^{1,1}_\delbar(M,J_\alpha,\omega_\alpha)=b^-+1=2$, and $\real(u)\ne\alpha r^2$ if and only if $h^{1,1}_\delbar(M,J_\alpha,\omega_\alpha)=b^-=1$.

%Let us also check if there exist almost K\"ahler metrics on $(M,J)$.
Note that $\omega_\alpha$ is almost K\"ahler, i.e., $\delbar\omega_\alpha=0$, if and only if
\begin{equation*}
\real(u)=\alpha r^2.
\end{equation*}
%Therefore, by Proposition \ref{prop-alm-kahl}, there exist no almost K\"ahler metrics on $(M,J)$.

Summing up, we have just proved the following
\begin{proposition}
Let $M=\Gamma\backslash G$, $J_\alpha$, $\omega_\alpha$ be as above. We have $h^{1,1}_\delbar(M,J_\alpha,\omega_\alpha)=b^-+1=2$ iff $\real(u)=\alpha r^2$, otherwise $h^{1,1}_\delbar(M,J_\alpha,\omega_\alpha)=b^-=1$.
Moreover, $\omega_\alpha$ is almost K\"ahler iff $\real(u)=\alpha r^2$.
\end{proposition}

\subsection{Primary Kodaira surface, II}

Let $M=\Gamma\backslash G$ be the primary Kodaira surface, as in the previous subsection.

We endow $G$ and $M$ with the left invariant almost complex structure $J_\beta$ given by
\[
\phi^1=e^4+ie^1,\ \ \ \phi^2=e^2-i\beta e^3,
\]
with $\beta\in\R\setminus\{0\}$,
being a global coframe of the vector bundle of $(1,0)$ forms $T^{1,0}G$. The associated structure equations are
\begin{align*}
&d\phi^1=0,\\
&d\phi^2=\frac{\beta}4\big(\phi^{12}+\phi^{1\c2}+\phi^{2\c1}-\phi^{\c1\c2}\big),
\end{align*}
therefore the almost complex structure $J_\beta$ is non integrable.
This is the same almost complex structure considered in \cite[Section 6]{CZ}, \cite[Section 2]{HZ} and \cite[Section 5]{PT4}.
From the structure equations we derive
\begin{align*}
\frac4\beta\del\phi^{1\c1}&=0,&
\frac4\beta\delbar\phi^{1\c1}&=0,\\
\frac4\beta\del\phi^{1\c2}&=\phi^{12\c1},&
\frac4\beta\delbar\phi^{1\c2}&=-\phi^{1\c1\c2},\\
\frac4\beta\del\phi^{2\c1}&=\phi^{12\c1},&
\frac4\beta\delbar\phi^{2\c1}&=-\phi^{1\c1\c2},\\
\frac4\beta\del\phi^{2\c2}&=0,&
\frac4\beta\delbar\phi^{2\c2}&=0.
\end{align*}

Endow $(G,J_\beta)$ with a left invariant almost Hermitian metric
\begin{equation*}%\label{eq-omega}
\omega_\beta=ir^2\phi^{1\c1}+is^2\phi^{2\c2}+u\phi^{1\c2}-\c{u}\phi^{2\c1},
\end{equation*}
with $u\in\C$, $r,s\in\R$, $r,s>0$ and $rs>|u|$. Set $\tau=\sqrt{r^2s^2-|u|^2}\ne0$.

By Corollary \ref{cor-char}, we have that $h^{1,1}_\delbar(M,J_\beta,\omega_\beta)=b^-+1$ if and only if there exist $A,B,C\in\C$ such that
\begin{numcases}{}
u+iAu-B\tau-\c{u}-iA\c{u}-C\tau=0,\label{eq-kod2-1}\\
u-iAu+B\tau-\c{u}+iA\c{u}+C\tau=0.\label{eq-kod2-2}
\end{numcases}
Summing \eqref{eq-kod2-1} and \eqref{eq-kod2-2}, subtracting \eqref{eq-kod2-1} from \eqref{eq-kod2-2}, we obtain
\begin{equation*}
\begin{cases}
u=\c{u},\\
B=-C.
\end{cases}
\end{equation*}
Therefore, $\img(u)=0$ if and only if $h^{1,1}_\delbar(M,J_\beta,\omega_\beta)=b^-+1=2$, and $\img(u)\ne0$ if and only if $h^{1,1}_\delbar(M,J_\beta,\omega_\beta)=b^-=1$.

%Let us also check if there exist almost K\"ahler metrics on $(M,J)$.
Note that $\omega_\beta$ is almost K\"ahler, i.e., $\delbar\omega_\beta=0$, if and only if
\begin{equation*}
\img(u)=0.
\end{equation*}
%Therefore, by Proposition \ref{prop-alm-kahl}, there exist no almost K\"ahler metrics on $(M,J)$.

Summing up, we have just proved the following
\begin{proposition}
Let $M=\Gamma\backslash G$, $J_\beta$, $\omega_\beta$ be as above. We have $h^{1,1}_\delbar(M,J_\beta,\omega_\beta)=b^-+1=2$ iff $\img(u)=0$, otherwise $h^{1,1}_\delbar(M,J_\beta,\omega_\beta)=b^-=1$.
Moreover, $\omega_\beta$ is almost K\"ahler iff $\img(u)=0$.
\end{proposition}

\end{document}